\documentclass{amsart}
\usepackage{graphicx}
\usepackage{mathtools,etoolbox}
\usepackage{verbatim}
\usepackage[left=3.5cm, right=3.5cm]{geometry}
\usepackage{array, boldline, booktabs, makecell}

\usepackage[svgnames, table]{xcolor}

\usepackage{ifthen}
\usepackage[T1]{fontenc}
\usepackage[utf8]{inputenc}
\usepackage[all]{xy}
\usepackage{graphicx}
\usepackage{enumitem}
\usepackage{xspace}
\usepackage{pdfsync}
\usepackage{epic}
\usepackage{dsfont}
\usepackage{tikz-cd}
\usepackage{amssymb}
\usepackage{stmaryrd}

\usepackage{multirow}

\usepackage{MnSymbol}
\usepackage{accents}

\usepackage{hyperref}
\usepackage{varioref}
\usepackage[hyperpageref]{backref}

\newcommand{\norm}[1] {\| #1 \| }
\usepackage{xcolor}
\hypersetup{
    colorlinks,
    linkcolor={red!60!black},
    citecolor={blue!60!black},
    urlcolor={blue!50!black}
}

\theoremstyle{theorem} \newtheorem{theorem}{Theorem}
\newtheorem{proposition}{Proposition} [section]
\newtheorem{lemma}[proposition]{Lemma}
\newtheorem{definition}[proposition]{Definition}

\newtheorem{conjecture}[proposition]{Conjecture}

\theoremstyle{definition} 
\theoremstyle{remark} 
\theoremstyle{remark} \newtheorem{remark}[proposition]{Remark}
\theoremstyle{definition} 
\theoremstyle{definition}

\newcommand{\quotientd}[2]{{\left.\raisebox{.2em}{$#1$}\middle\slash\raisebox{-.2em}{$#2$}\right.}}

\def\det{\mathop{\rm det}\nolimits}

\newcounter{par}[subsection]
\newcounter{parsub}[subsection]

	\title[Hyperbolicity of generic hypersurfaces]{Hyperbolicity of generic hypersurfaces of polynomial degree via Green-Griffiths jet differentials}

\author[B. Cadorel]{Beno\^{i}t Cadorel} 
\email{benoit.cadorel@univ-lorraine.fr}
\address{Institut \'Elie Cartan de Lorraine, Universit\'e de Lorraine, F-54000 Nancy,
	France.}
\urladdr{http://www.normalesup.org/~bcadorel/} 

\begin{document}
\maketitle

\begin{abstract}
	We give a new version of a recent result of B\'{e}rczi-Kirwan, proving the Kobayashi and Green-Griffiths-Lang conjectures for generic hypersurfaces in \(\mathbb{P}^{n+1}\), with a polynomial lower bound on the degree. Our strategy again relies on Siu's technique of slanted vector fields and the use of holomorphic Morse inequalities to prove the existence of a jet differential equation with a negative twist -- however, instead of using a space of invariant jet differentials, we base our computations on the classical Green-Griffiths jet spaces.
\end{abstract}

\section{Introduction}

The Green-Griffiths-Lang conjecture \cite{GG80,Lan87} predicts that any projective manifold of general type \(X\) should be {\em quasi-Brody hyperbolic}, namely there should exist a proper algebraic subset \(Z \subsetneq X\) containing the image of any non-constant holomorphic map \(f : \mathbb{C} \longrightarrow X\). Studying this question in the case of hypersurfaces is already a very difficult problem -- in this situation, Kobayashi also conjectured that one should be able to obtain genuine Brody hyperbolicity for higher degrees and {\em generic} hypersurfaces:

\begin{conjecture} \label{conj:GGL} Let \(n \geq 2\) be an integer.
	\begin{enumerate}
		\item {\em (Green-Griffiths-Lang conjecture for smooth hypersurfaces)} Any smooth hypersurface \(X \subset \mathbb{P}^{n+1}\) of degree \(d \geq n + 3\) is quasi-Brody hyperbolic;
		\item {\em (Kobayashi conjecture \cite{kob70})} 
			\footnote{The bound in the second item did not appear in Kobayashi's original article: it would follow naturally from results by Clemens-Ein-Voisin-Pacienza (see \cite{cle86, ein88, ein91, voi96, voi98, pac04}) -- at least for a very general hypersurface -- if the Green-Griffiths-Lang conjecture were known to hold in full generality.}
			A {\em generic} hypersurface \(X \subset \mathbb{P}^{n+1}\) of degree \(d \geq 2n\) (or \(d \geq 2n+1\) if \(n \leq 4\)) is {\em Brody hyperbolic} i.e. there exists no non-constant holomorphic map \(f : \mathbb{C} \to X\).
	\end{enumerate}
\end{conjecture}

The previous conjecture has attracted a lot of attention in the last few years, and we now know that the two items hold if we consider {\em generic} hypersurfaces of high enough degree:

\begin{theorem} \label{thm:DMR} There exists two sequences of integers \(d_{n}, d_{n}'\) such that the following hold:
	\begin{enumerate}
		\item {\em (Diverio-Merker-Rousseau \cite{DMR10})} a generic hypersurface \(X \subset \mathbb{P}^{n+1}\) of degree \(d \geq d_{n}\) is quasi-Brody hyperbolic;
		\item  {\em (Brotbek \cite{Bro15})} a generic hypersurface \(X \subset \mathbb{P}^{n+1}\) of degree \(d \geq d_{n}'\) is Brody hyperbolic.
	\end{enumerate}
\end{theorem}

Obtaining effective bounds for the sequences \(d_{n}, d_{n}'\) is no easy task, and we are still very far from the linear bounds of Conjecture~\ref{conj:GGL}. However, a significant breakthrough has been made recently by B\'{e}rczi and Kirwan \cite{BK24}, who managed to obtain polynomial bounds \(d_{n} \approx d_{n}' \approx O(n^{4})\) in both items -- thus substantially improving the previously known bounds, that were all growing at least as \(e^{O(n\log n)}\) (see e.g. \cite{Den20, Dem20, MT19}).
\medskip

The strategy of \cite{BK24} relies on the technique of {\em slanted vector fields} introduced by Siu \cite{siu04}. Eventually, everything boils down to proving the bigness of a well-chosen line bundle on an projective {\em jet space} \(X_{k} \to X\) sitting above \(X\) (see Section~\ref{sec:strategyslanted} below). The classical strategy to prove this bigness is to apply Siu's {\em algebraic Morse inequalities}, that requires in turn to show the positivity of an adequate intersection number.

There are several possible choices for the jet space \(X_{k}\), but not all seem to give very satisfactory bounds on \(d_{n}\) or \(d_{n}'\): most of the previous exponential bounds were obtained for example using the Demailly-Semple jets spaces \(X_{k} = X_{k}^{DS}\). The novelty in \cite{BK24} was to introduce a jet space \(X_{k} = X_{k}^{BK}\) on which the intersection theory is much more favorable, by means of the {\em non-reductive Geometric Invariant Theory}.
\medskip

The previous spaces \(X_{k}^{DS}\) and \(X_{k}^{BK}\) are jet bundles naturally associated to the so-called {\em invariant} jet differentials -- their definition is quite elaborate compared to the {\em Green-Griffiths jet bundles} \(X_{k}^{GG}\) introduced more than 40 years ago (see \cite{GG80}). Quite surprisingly, these latter jet spaces seem to have been a bit overlooked in their potential applications to the problem at hand.
\medskip

In these notes, we will show that it is indeed possible to use \(X_{k} = X_{k}^{GG}\) and that following the strategy described above also yields polynomial degree bounds. More precisely, one can show the following:
\begin{theorem} \label{thm:poly} Let \(n \geq 2\) be an integer.
	\begin{enumerate}
		\item a generic hypersurface \(X \geq \mathbb{P}^{n+1}\) of degree \(d > \frac{153}{4} n^{5}\) is quasi-Brody hyperbolic;
		\item a generic hypersurface \(X \geq \mathbb{P}^{n+1}\) of degree \(d > \frac{153}{4}(2n-1)^{5}\) is Brody hyperbolic.
	\end{enumerate}
\end{theorem}

The fact that the bound looks similar in the second item is no mystery: as in \cite{BK24}, it follows from the work of Riedl-Yang \cite{RY22} that if the first item of Theorem~\ref{thm:DMR} has been proved using e.g. the jet differentials techniques of \cite{DMR10}, then the second item must also hold with \(d_{n}' = d_{2n-1}\).
\medskip

As we explained above, we make no change to the strategy of slanted vector fields: the only new input is the computation of the intersection number coming from the algebraic Morse inequalities. To perform these computations, we will use the theory of {\em weighted projective bundles} and their associated Segre classes, in a manner very similar to some earlier work of the author on jet differentials on compactifications of ball quotients (see \cite{Cad16}).

\subsection{Organization of the article.} These notes are divided in three parts and an annex:
\begin{enumerate}
	\item Section~\ref{sec:weighted}: we gather a few facts on weighted vector bundles, jets spaces and the holomorphic Morse inequalities;
	\item Section~\ref{sec:strategyslanted}: we recall the main criterion for hyperbolicity of generic hypersurfaces, that sums up the strategy of slanted vector fields (see Theorem~\ref{thm:hypslanted}). We then present the positivity statement that is needed to apply the holomorphic Morse inequalities (Proposition~\ref{prop:mainresult}).
	\item Section~\ref{sec:maincomputations}: we prove Proposition~\ref{prop:mainresult}.
	\item Section~\ref{sec:annex}: In an annex to this article, we give a quite elementary proof of the numerical version of the Whitney formula employed in Section~\ref{sec:maincomputations} (see the equation~\eqref{eq:whitney}). The author hopes this proof is even simpler than the one he presented in his thesis; in the end, it is based on straightforward computations of integrals on simplexes (in a manner very similar to the seminal work of Green-Griffiths \cite{GG80}).
\end{enumerate}

\subsection{Acknowledgments} The author wishes to thank Damian Brotbek, Frédéric Campana, Lionel Darondeau, Simone Diverio, Antoine \'{E}tesse, Joël Merker, Eric Riedl and Erwan Rousseau for all the discussions that took place before and during the preparation of this work. Special thanks are due to Gergely B\'{e}rczi and Frances Kirwan for their kind and enlightening explanations on non-reductive GIT, during our stay at the Isaac Newton Institute of Cambridge. 
\medskip

During the preparation of this work, the author was supported by the French ANR project KARMAPOLIS (ANR-21-CE40-0010). The author would also like to thank the Isaac Newton Institute for Mathematical Sciences for the support and hospitality during the programme {\em New equivariant methods in algebraic and differential geometry} when work on this paper was undertaken. This work was supported by: EPSRC Grant Number EP/R014604/1. This work was also partially supported by a grant from the Simons Foundation.

\section{Weighted projective bundles and Green-Griffiths jet differentials} \label{sec:weighted}

We recall here some of the results and notation of \cite{Cad16} and \cite[Chapitre 3]{cad_thesis} pertaining to weighted projective bundles and their intersection theory.

\subsubsection{Weighted projective bundles} Let \(X\) be a complex projective manifold. By a {\em weighted vector bundle} on \(X\), we mean the data of finitely many couples \((E_{i}, a_{i})_{1 \leq i \leq s}\), where \(E_{i} \to X\) are vector bundles, and the \(a_{i} \geq 1\) are integers. We will often write this data under the form
\[
	E_{1}^{(a_{1})} \oplus \dotsc \oplus E_{s}^{(a_{s})}.
\]

Given a weighted vector bundle, we can construct several associated objects on \(X\):

\begin{definition} \label{def:symweight} Let \(\mathbf{E} := E_{1}^{(a_{1})} \oplus \dotsc \oplus E_{s}^{(a_{s})}\) be a weighted vector bundle over \(X\). 
	\begin{enumerate}
		\item The {\em dual} of \(\mathbf{E}\) is \(\mathbf{E}^{\ast} := (E_{1}^{\ast})^{(a_{1})} \oplus \dotsc \oplus (E_{s}^{\ast})^{(a_{s})}\);
		\item The {\em symmetric algebra} of \(\mathbf{E}\) is the graded \(\mathcal{O}_{X}\)-algebra \(S^{\bullet} \mathbf{E} = \bigoplus_{m \in \mathbb{N}} S^{m} \mathbf{E}\) whose pieces are the vector bundles
			\[
				S^{m}(E_{1}^{(a_{1})} \oplus \dotsc \oplus E_{s}^{(a_{s})})
				:=
				\bigoplus_{a_{1} l_{1} + \dotsc + a_{s} l_{s} = m}
				S^{l_{1}} E_{1} \otimes \dotsc \otimes S^{l_{s}} E_{s},
			\]
			endowed with its natural product law \((S^{l_{1}} E_{1} \otimes \dotsc \otimes S^{l_{s}} E_{s}) \otimes (S^{l_{1}'} E_{1} \otimes \dotsc \otimes S^{l_{s}'} E_{s}) \longrightarrow S^{l_{1} + l_{1}'} E_{1} \otimes \dotsc \otimes S^{l_{s} + l_{s}'} E_{s}\).

		\item the {\em weighted projective space} \(\mathbb{P}(\mathbf{E})\) is the projectivized scheme
			\[
				\mathbb{P}(\mathbf{E})
				=
				\mathrm{Proj}_{X} (S^{\bullet} \mathbf{E}).
			\]
	\end{enumerate}
\end{definition}

With the notation of the previous definition, one can also define dually \(\mathbb{P}(\mathbf{E})\) as a \(\mathbb{C}^{\ast}\)-quotient:
\[
	P(\mathbf{E}^{\ast}) = P(E_{1}^{\ast}{}^{(a_{1})} 
	\oplus 
	\dotsc
	\oplus E_{s}^{\ast}{}^{(a_{s})})
	:=
	\quotientd{\left(E_{1}^{\ast} \oplus \dotsc \oplus E_{s}^{\ast}\right) - \{0\}}{\mathbb{C}^{\ast}},
\]
where by \(\{0\}\) we denote the zero section, and the action of \(\lambda \in \mathbb{C}^{\ast}\) on the total space of \(\mathbf{E}^{\ast}\) is given fiberwise by
\[
	\lambda \cdot (v_{1}, \dotsc, v_{r})
	=
	(\lambda^{a_{1}} v_{1}, \dotsc, \lambda^{a_{s}} v_{s}).
\]

This implies that \(\pi: \mathbb{P}(\mathbf{E}) \to X\) is a bundle in weighted projective spaces ; it is endowed with tautological {\em sheaves} \(\mathcal{O}^{\mathrm{sh}}(m)\) (with \(m \in \mathbb{N}\)) for which one has
\[
	\pi_{\ast} \mathcal{O}^{\mathrm{sh}}(m) = S^{m}(\mathbf{E}).
\]
In general, the \(\mathcal{O}^{\mathrm{sh}}(m)\) are not locally trivial, but this is however the case if \(\mathrm{gcd}(a_{1}, \dotsc, a_{s})\) divides \(m\). If one has also \(m>0\), then \(\mathcal{O}^{\mathrm{sh}}(m)\) is a relatively ample line bundle with respect to \(\pi\). In all the following, we will use the notation \(\mathcal{O}(1)\) to denote the \(m\)th-root of \(\mathcal{O}^{\mathrm{sh}}(m)\) as a \(\mathbb{Q}\)-line bundle, i.e. the element \(\mathcal{O}(1) := \frac{1}{m} \mathcal{O}^{\mathrm{sh}}(m)\) in the rational Picard group \(\mathrm{Pic}\, \mathbb{P}(\mathbf{E}) \otimes \mathbb{Q}\). Accordingly, we let \(\mathcal{O}(d) := \mathcal{O}(1)^{\otimes d}\) for any integer \(d \geq 1\); this element coincides with \(\mathcal{O}^{\mathrm{sh}}(d)\) (up to \(\mathbb{Q}\)-linear equivalence) if \(d\) is divisible enough.

\begin{remark}
	Alternatively, one could also see \(\mathbb{P}(\mathbf{E})\) as a smooth Deligne-Mumford stack, endowed with a natural tautological (stacky) line bundle \(\mathcal{O}(1)\). In this case, the bundles \(\mathcal{O}(m)\) can all be seen as line bundles on the corresponding stack, and one has naturally \(\mathcal{O}(1)^{\otimes m} = \mathcal{O}(m)\). 
\end{remark}

\subsubsection{Weighted Segre classes} If \(\mathbf{E} = E_{1}^{(a_{1})} \oplus \dotsc \oplus E_{s}^{(a_{s})}\) is a weighted vector bundle, we gave in \cite{Cad16} a definition of the Segre classes of \(E\) as endomorphisms of the rational Chow rings \((A_{\ast} X)_{\mathbb{Q}}\), as follows. Let \(\alpha \in A_{\ast} X\) be any class on \(X\). Then one lets
\begin{equation} \label{eq:segreclass}
	s_{j}(\mathbf{E}) \cap \alpha = \frac{1}{m^{j + r - 1}} \pi_{\ast} (c_{1} \mathcal{O}(m)^{j + r -1} \cap \pi^{\ast} \alpha).
\end{equation}
where \(\pi : \mathbb{P}(\mathbf{E}^{\ast}) \to X\) is the natural projection, \(r = \sum_{j} \mathrm{rk} E_{j}\) and \(m := \mathrm{lcm}(a_{1}, \dotsc, a_{s})\)\footnote{In \cite{Cad16}, we used the notation \(r = \sum_{j} \mathrm{rk} E_{j} - 1\) instead.} . In this situation, we proved the following Whitney formula in \cite[Proposition~3.2.11]{cad_thesis}, that makes sense as an equality between endomorphisms of \((A_{\ast} X)_{\mathbb{Q}}\):
\begin{equation} \label{eq:whitney}
	s_{\bullet}(E_{1}^{(a_{1})} \oplus \dotsc \oplus E_{s}^{(a_{s})})
	=
	\frac{\mathrm{gcd}(a_{1}, \dotsc, a_{s})}{a_{1} \dotsc a_{s}}
	\prod_{j} s_{\bullet}(E_{j}^{(a_{j})}), \;
\end{equation}
where
	\(
	s_{\bullet}(E^{(a)}) 
	= 
	\frac{1}{a^{\mathrm{rk} E - 1}} \sum_{l} \frac{s_{l}(E)}{a^{l}}
	\)
	for a single vector bundle \(E\) and any integer \(a > 0\). The reader can refer to the annex (see Section~\ref{sec:annex}) for a proof of a numerical version of this formula, based on straightforward computations of Euler characteristics.

\subsubsection{Green-Griffiths jet differentials} Let \(X\) be a complex projective manifold. We refer to \cite[Section~7]{Dem12a} for all definitions related to Green-Griffiths jet differentials. For our purposes, it will be enough to know that for any order \(k \in \mathbb{N}\), we may define the {\em Green-Griffiths algebra of holomorphic jet differentials}
\[
	E_{k, \bullet}^{GG} \Omega_{X} = \bigoplus_{m \in \mathbb{N}} E_{k, m}^{GG} \Omega_{X}.
\]
which is an \(\mathcal{O}_{X}\)-algebra whose sections represent holomorphic differentials equations of order \(k\) on \(X\). The Green-Griffiths jet bundles are the projective schemes associated to these algebras:
\[
	X_{k}^{GG} := \mathbf{Proj}_{X}(E_{k, \bullet}^{GG} \Omega) \overset{p_{k}}{\longrightarrow} X
\]
These spaces are endowed with natural tautological \(\mathbb{Q}\)-line bundles \(\mathcal{O}_{GG, k}(1)\) such that
\[
	E_{k, m}^{GG} \Omega_{X} = p_{k}^{\ast} \mathcal{O}_{GG, k}(m).
\]

One of the crucial properties of the Green-Griffiths algebra is the existence of a natural filtration whose graded object is easy to describe in terms of weighted vector bundles.

\begin{theorem}[Green-Griffiths \cite{GG80}, see also {\cite[Section~7]{Dem12a}}]
	Let \(k \in \mathbb{N}\). Then there exists a filtration \(F\) on \(E_{k, \bullet}^{GG} \Omega_{X}\), compatible with its structure of \(\mathcal{O}_{X}\)-algebra, and whose associated graded algebra satisfies 
	\[
		\mathrm{Gr}_{F}(E_{k, \bullet}^{GG} \Omega_{X}) \cong S^{\bullet} \mathbf{\Omega}_{k} ,
	\]
	where \(\mathbf{\Omega}_{k} := \Omega_{X}^{(1)} \oplus \dotsc \oplus \Omega_{X}^{(k)}\).
\end{theorem}

By elementary considerations on short exact sequences (\cite[Lemma~2.15]{Mer15}, see also \cite[Proposition~2.2]{Cad19}), one has, for any line bundle \(L \to X\):
\[
	h^{0}(E_{k, m}^{GG} \Omega_{X} \otimes L) \geq 
	h^{0}(X, S^{m}\mathbf{\Omega}_{k} \otimes L)  
	- 
	h^{1}(X, S^{m}\mathbf{\Omega}_{k} \otimes L).
\]

We will use this inequality jointly with the following result:

\begin{proposition} Let \(P_{k} := \mathbb{P}(\mathbf{\Omega}_{k}) \overset{\pi_{k}}{\longrightarrow} X\), and let \(L\) be a line bundle on \(X\). Then one has, for any \(1 \leq i \leq \dim X\) and any \(m \geq 1\) divisible by \(\gcd(1, 2, \dotsc, k)\):
\[
	h^{i}(X, S^{m} {\bf \Omega}_{k} \otimes L) 
	= 
	h^{i}(P_{k}, \mathcal{O}_{P_{k}}(m) \otimes \pi_{k}^{\ast} L),
\]
	where \(\mathcal{O}_{P_{k}}(m)\) are the tautological line bundles on \(P_{k}\).
\end{proposition}
\begin{proof}
	By the Leray spectral sequence and the projection formula, it suffices to show that \(R^{i} (\pi_{k})_{\ast} \mathcal{O}_{P_{k}}(m) = 0\) for any \(i > 0\) and any \(m\) divisible enough. This can be checked fiberwise, and immediately follows from the corresponding results for the cohomology of weighted projective spaces (see e.g. \cite[1.4~Theorem]{dol82}).
\end{proof}

In particular, if we consider a very ample line bundle \(\mathcal{O}_{X}(1)\) on \(X\) and any positive \(\epsilon > 0\), one has
\begin{equation} \label{eq:lowerboundhOh1}
	h^{0}(X_{k}^{GG}, \mathcal{O}_{GG, k}(m) \otimes p_{k}^{\ast} \mathcal{O}(-m \epsilon))
	\geq
	(h^{0} - h^{1})(P_{k}, \mathcal{O}_{P_{k}}(m) \otimes \pi_{k}^{\ast} \mathcal{O}(-m \epsilon))
\end{equation}

\subsubsection{Holomorphic Morse inequalities}

Let us recall the statement of the famous holomorphic Morse inequalities, in the version proved by Siu:

\begin{theorem} \label{thm:HMI}
	[{Siu~\cite{siu93}, Demailly~\cite[§12]{dem96}, see also \cite[Remark 2.2.20]{laz04}}]
	Let \(Y\) be a complex projective variety of dimension \(n\). Let \(A, B\) be two nef line bundles on \(Y\), and let \(L := A \otimes B^{-1}\). Then one has
	\[
		(h^{0} - h^{1}) (Y, L^{\otimes m})
		\geq 
		(A^{n} - n A^{n-1} \cdot B) \frac{m^{n}}{n!}
		+
		O(m^{n-1}).
	\]
\end{theorem}

\subsubsection{Nefness of adequate twists} With the notation of the previous section, our wish in Section~\ref{sec:strategyslanted} will be to apply the holomorphic Morse inequalities to a line bundle of the form \(\mathcal{O}_{P_{k}}(m) \otimes \mathcal{O}_{X}(-m \epsilon)\). To do this, the following statement will be quite useful. 

\begin{proposition} \label{prop:nef}
	Let \(\mathcal{O}_{X}(1)\) be a very ample line bundle on \(X\). Then the \(\mathbb{Q}\)-line bundle \(L_{k} = \mathcal{O}_{P_{k}}(1) \otimes \pi^{\ast}_{k} \mathcal{O}_{X}(2)\) is nef on \(P_{k}\).
\end{proposition}

\begin{proof}
	It follows from Definition~\ref{def:symweight}~{\em (1)} that for any \(m \geq 0\), one has
	\[
		S^{m}(\mathbf{\Omega}_{k}) \otimes \mathcal{O}(2m)
		\cong
		S^{m}( \bigoplus_{1 \leq l \leq k} \Omega_{X}(2l)^{(l)})
	\]
	for any \(m \in \mathbb{N}\). This implies that \(L_{k}\) can be seen as the \(\mathbb{Q}\)-tautological line bundle of the weighted vector bundle
	\[
		\mathbb{P}(\bigoplus_{1 \leq l \leq k} \Omega_{X}(2l)^{(l)}),
	\]
	which is naturally isomorphic to \(P_{k}\) as a scheme above \(X\). However, Lemma~\ref{lem:globalgen} below implies that each of the pieces \(\Omega_{X}(2l)\) is globally generated. This implies that for \(m\) divisible enough, the line bundle \(L_{k}^{\otimes m}\) is globally generated as well, and hence is nef.
\end{proof}

The following very classical lemma was used in the proof of the previous proposition.
\begin{lemma} \label{lem:globalgen} Let \(X \subset \mathbb{P}^{n}\) be a submanifold, and let \(\mathcal{O}_{X}(1)\) be the associated very ample line bundle. Then \(\Omega_{X}(2)\) is globally generated.
\end{lemma}
\begin{proof}
	Since the restriction map \(\Omega_{\mathbb{P}_{n}}|_{X} \to \Omega_{X}\) is onto, it suffices to prove the result for \(X = \mathbb{P}^{n}\). Let \(Z_{0}, \dotsc, Z_{n}\) be homogeneous coordinates, and \(z_{i} := \frac{Z_{i}}{Z_{0}}\) be the associated inhomogeneous coordinates on the chart \(U_{0} := \{ Z_{0} \neq 0 \}\). Then the elements 
	\[
		Z_{0}^{2} dz_{i} = Z_{0} dZ_{i} - Z_{i} dZ_{0}
		\in
		H^{0}(\mathbb{P}^{n}, \Omega_{\mathbb{P}^{n}}(2))
	\]
	generate \(\Omega_{\mathbb{P}^{n}}(2)\) on \(U_{0}\). The same reasoning also holds for the other charts \(U_{j}\).
\end{proof}

\begin{remark}
	Using the semi-continuity of the {\em nef} property for the countable Zariski topology, we can see as in \cite[Proposition~4.4]{Cad16} that the \(\mathbb{Q}\)-line bundle \(\mathcal{O}_{GG, k}(1) \otimes p_{k}^{\ast} \mathcal{O}(2)\) is nef on \(X_{k}^{GG}\). The fact that we may obtain a nef line bundle on \(X_{k}^{GG}\) by taking a twist on the base {\em independently of \(k\)} is in stark contrast with the case of the Demailly-Semple tower, where we need an twist on the base growing exponentially fast as we climb the jet tower (see \cite[Proposition~7.19]{Dem12a}). As B\'{e}rczi-Kirwan remarked in \cite{BK24}, it is also possible to use a constant twist for their non-reductive GIT quotient \(X_{k}^{BK}\), which makes the holomorphic Morse inequalities much easier to satisfy. 
\end{remark}

\section{Statement of the main results} \label{sec:strategyslanted}

In this section, we state the main estimates which, joint with Siu's strategy of slanted vector fields and Riedl-Yang's work \cite{RY22}, allow to derive the main result. All of this has become quite classical, so we will only quote the necessary statements, and refer to the original articles for more details.
\medskip

We fix an integer \(n \geq 2\). The general slanted jet techniques give the following: 

\begin{theorem} [\cite{DMR10, mer09, dar16}] \label{thm:hypslanted}
	Fix \(d \geq n\).  Assume that for any smooth hypersurface \(X \subset \mathbb{P}^{n+1}\) of degree \(d\), we have proven that the \(\mathbb{Q}\)-line bundle
	\[
		\mathcal{O}_{n, GG}(1)  \otimes p_{n}^{\ast} \mathcal{O}_{X} (- (5n +3)))
	\]
	is big. Then the generic hypersurface of degree \(d\) is quasi-Brody hyperbolic.
\end{theorem}

Let \(\epsilon > 0\) and \(k \in \mathbb{N}_{\geq 1}\) to be fixed later. To prove the bigness of the \(\mathbb{Q}\)-line bundle \(\mathcal{O}_{k, GG}(1)  \otimes p_{k}^{\ast} \mathcal{O}_{X} (- \epsilon)\), we see in view of \eqref{eq:lowerboundhOh1} and Theorem~\ref{thm:HMI} that it suffices to apply the Morse inequalities to the following \(\mathbb{Q}\)-line bundle on \(P_{k} = \mathbb{P}(\mathbf{\Omega}_{k})\):
\[
	M := \mathcal{O}_{P_{k}}(1) \otimes \pi_{k}^{\ast} \mathcal{O}_{X}(-\epsilon)
\]
To do this, use first Proposition~\ref{prop:nef} to write it as a difference of two nef line bundles
\[
	M = A - B,
\]
with \(A = \mathcal{O}_{P_{k}^{GG}}(1) \otimes \pi_{k}^{\ast} \mathcal{O}_{X}(2)\) and \(B = p_{k}^{\ast} \mathcal{O}_{X}(2 + \epsilon)\). The Morse inequalities then ask to show the positivity of 
\[
	P(n, d, \epsilon) := A^{N_{k}} - N_{k} A^{N_{k}} \cdot B.
\]
where \(N_{k} = \dim P_{k} = n + nk - 1\). To satisfy the hypothesis of Theorem~\ref{thm:hypslanted}, we will just have to specialize to the case \(k = n\) and \(\epsilon = 5n + 3\).
\medskip

This positivity of \(P(n, d, \epsilon)\) can be achieved thanks to the following proposition, that will be proved in Section~\ref{sec:maincomputations}. The idea of using the Fujiwara estimates (see Lemma~\ref{lem:fujiwara}) originally stems from \cite{DMR10} and has been used again e.g. in \cite{dar16a, BK24}.

\begin{proposition} [Main estimates] \label{prop:mainresult} Assume \(n \geq 2\), and fix \(k = n\). Let \(\epsilon > 0\) be a rational number.
	\begin{enumerate}
		\item One may write
			\begin{equation} \label{eq:polynomial}
				P(n, d, \epsilon) = d \sum_{j=0}^{n} Q_{j}(n, \epsilon) d^{j}.
			\end{equation}
			for some polynomials with rational coefficients \(Q_{j}(n, \epsilon)\). The leading term \(Q_{n}(n, \epsilon) > 0\) actually depends only on \(n\).
		\item There is a number \(D_{\epsilon} > 0\), depending only on \(\epsilon\) such that
			\begin{equation*}
				| Q_{j}(n, \epsilon) | < (D_{\epsilon} n^{4})^{n-j} Q_{n}(n, \epsilon) \label{eq:fujiwarabound}
			\end{equation*}
			for all \(j \geq 1\). One may actually take \(D_{\epsilon} = \max(\frac{27}{2}, 9(1 + \frac{\epsilon}{4}))\).
		\item (Fujiwara bound) As a consequence, if
			\[
				d > 2 D_{\epsilon}\cdot n^{4}
			\]
			then \(P(n, d, \epsilon) > 0\). For such values of \(n, d, \epsilon\), the line bundle
			\[
				\mathcal{O}_{X_{n}^{GG}}(1) \otimes \pi^{\ast}_{n} \mathcal{O}_{X}(-\epsilon)
			\]
			is big.
	\end{enumerate}
\end{proposition}

Thus, if we fix first \(n \geq 2\), and then take \(\epsilon := 5n+3\) in Proposition~\ref{prop:mainresult}, we get the bound
\[
	d > 18(\frac{5n+3}{4} + 1)n^{4}.
\]

Let us simply take a monomial lower bound that ensures positivity for all \(n \geq 2\):

\begin{theorem}[=Theorem~\ref{thm:poly} {\em (1)}]
	For \(n\geq 2\) and \(d > \frac{153}{4} n^{5}\), the generic hypersurface of degree \(d\) in \(\mathbb{P}^{n+1}\) is quasi-hyperbolic.
\end{theorem}

This proves the first item of Theorem~\ref{thm:poly}. The second one follows from the work of Riedl-Yang \cite{RY22}, who showed that if the first item has been proven for \(d \geq d_{n}\) using the jet differential techniques discussed here, then the Kobayashi conjecture must hold with the lower bound \(d_{n}' = d_{2n-1}\).

\section{Main computations} \label{sec:maincomputations}

In this section, we prove Proposition~\ref{prop:mainresult}. We will first drop the hypothesis \(k = n\) to perform the beginning of our computations; we will only resume this hypothesis after Step 4.
\medskip

\noindent
{\em Step 1. Expression of \(A^{N_{k}}\) and \(A^{N_{k}-1} \cdot B\) in terms of weighted Segre classes.} Recall that \(P_{k} = \mathbb{P}(\mathbf{\Omega}_{k})\). Let us denote by \(\mathcal{O}_{k}(1)\) its tautological \(\mathbb{Q}\)-line bundle. Dually, we may also write \(P_{k} = \mathrm{P}(\mathbf{T}_{k})\) with
\[
	\mathbf{T}_{k}
	:=
	T_{X}^{(1)}
	\oplus \dotsc \oplus
	T_{X}^{(k)}.
\]
Thus, one may expand the Newton binomial and use the definition of weighted Segre classes \eqref{eq:segreclass} to obtain the following (pullbacks to \(P_{k}\) are implied):
\begin{align} \nonumber
	A^{N_{k}} & = \int_{P_{k}} ( u + 2 h)^{N_{k}}
	\hspace{2cm} (u = c_{1}(\mathcal{O}_{k}(1)), h = c_{1}(\mathcal{O}_{X}(1))\,) \\ \label{eq:A}
	& = \sum_{l=0}^{n} 2^{l} \binom{N_{k}}{l} \int_{X} s_{n-l}(\mathbf{T}_{k}) h^{l},
\end{align}
On the other hand, we also have:

\begin{align} \nonumber
	A^{N_{k}-1} \cdot B & = \int_{P_{k}} ( u + 2 h)^{N_{k}-1} (2 + \epsilon) h \\ \label{eq:AB}
	& = \sum_{l=1}^{n} 2^{l-1}(2 + \epsilon) \binom{N_{k}-1}{l-1} \int_{X} s_{n-l}(\mathbf{T}_{k}) h^{l}.
\end{align}

\begin{comment}
Thus, singling out the \(0\) term in \(A^{N_{k}}\), one deduces that
\begin{align*}
	A^{N_{k}} - N_{k} A^{N_{k} - 1} \cdot B
	& = \int_{X} s_{n}(\mathbf{T}_{k})
	+
	\sum_{l=1}^{n} 2^{l-1} \left[2 \binom{N_{k}}{l} - (2 + \epsilon) N_{k} \binom{N_{k} - 1}{l-1}\right] \int_{X} s_{n-l}(\mathbf{T}_{k}) h^{l}\\
	& = \int_{X} s_{n}(\mathbf{T}_{k})
	+
	\sum_{l=1}^{n} 2^{l-1}(2 - (2 + \epsilon)l) \binom{N_{k}}{l} \int_{X} s_{n-l}(\mathbf{T}_{k}) h^{l}.
\end{align*}
\end{comment}

\noindent
{\em Step 2. Application of the Whitney formula.} To compute the previous numbers, we will apply the Whitney formula \eqref{eq:whitney} to express the weighted Segre classes of \(\mathbf{T}_{k}\) in terms of the hyperplane class \(h\). In our context, since \(X \subset \mathbb{P}^{n+1}\) is a smooth degree \(d\) hypersurface, we have 
\[
	s_{\bullet}(T_{X}) 
	= s_{\bullet}(T_{\mathbb{P}^{n+1}}) c_{\bullet}(N_{X}) 
	= [ \sum_{j=0}^{n} (-h)^{j}]^{n+2} (1 + hd).
	\]
Thus the Whitney formula yields
\begin{equation} \label{eq:segre}
	s_{\bullet}(\mathbf{T}_{k})
	=
	\frac{1}{(k!)^{n}}
	\prod_{1 \leq l \leq k} 
	\left[	
	\sum_{j=0}^{n} (- \frac{h}{l})^{j}
	\right]^{n+2}
	\left[
		1 + \frac{hd}{l}
	\right]
\end{equation}

\noindent
{\em Step 3. Expressions of \eqref{eq:A} and \eqref{eq:AB} as polynomials in \(d\)}. Let us start by writing the coefficient \(\Lambda_{\alpha, \beta}\) of \(d^{\alpha} h^{\beta}\) in the expression \(s_{\beta}(\mathbf{T}_{k})\) for any integers \(\alpha, \beta\). Inspection of \eqref{eq:segre} shows that \(\Lambda_{\alpha, \beta} = 0\) unless \(\beta \geq \alpha\). If this holds, let us write \(\beta = \alpha + \gamma\) with \(\gamma \geq 0\), in which case one then has
\[
	(k!)^{n} \Lambda_{\alpha, \beta} = (-1)^{\gamma} B_{\gamma} C_{\alpha}
	\quad
	(\gamma = \beta - \alpha \geq 0)
\]
where \(B_{\gamma}\) is the coefficient of \(h^{\gamma}\) in \(\prod_{1 \leq l \leq k} (\sum_{j=0}^{n} \frac{h}{j})^{n+2}\), and \(C_{\alpha}\) is the coefficient of \(h^{\alpha}\) in \(\prod_{1 \leq l \leq k}(1 + \frac{h}{l})\).

\begin{proposition} We have the following formulas:
\begin{enumerate}
	\item \(B_{0} = 1\) and for all \(\gamma \geq 1\), one has
	\[
		B_{\gamma} := \sum_{\{u_{1} \leq \dotsc \leq u_{\gamma}\} \subset S_{k, n+2}}
	\frac{1}{u_{1} \dotsc u_{\gamma}}
	\]
	In this expression, the sum runs over all non-decreasing sequences \(u_{1}, \dotsc, u_{\gamma}\) in the ordered set
\[
	S_{k, n+2} := \{
		1_{1} <  \dotsc < 1_{n+2} < 2_{1} < \dotsc < 2_{n+2} < \dotsc < k_{1} < \dotsc < k_{n+2}
	\},
\]
and the product is computed by forgetting the indexes (see also \cite[p.77]{cad_thesis}). 
\item \(C_{0} = 1\) and for all \(\alpha \geq 1\), one has
	\[
		C_{\alpha} := \sum_{1 \leq l_{1} < \dotsc < l_{\alpha} \leq k} \frac{1}{l_{1} \dotsc l_{\alpha}}
	\]
\end{enumerate}
\end{proposition}
\begin{proof}
	Only the first item needs explaining. If \(\gamma \geq 1\), the definition of \(B_{\gamma}\) gives
	\begin{align*}
		B_{\gamma} = \sum_{p_{1, 1} + \dotsc p_{1, n+2} + \dotsc + p_{k, 1} + \dotsc p_{k, n+2}=\gamma} 
		\frac{1}
		{1^{p_{1, 1}} \dotsc 1^{p_{1, n+2}} 2^{p_{2, 1}} \dotsc 2^{p_{2, n+2}} \dotsc k^{p_{k, 1}} \dotsc k^{p_{k, n+2}}}
	\end{align*}

	Now, there is a bijection between the set of \(k(n+2)\)-uples of integers \((p_{l, j})_{1 \leq l \leq k, 1 \leq j \leq n+2}\) summing to \(\gamma\), and the set of sequences \(\{ u_{1} \leq \dotsc \leq u_{\gamma}\} \subset S_{k, n+2}\): to any such tuple, we associate the sequence obtained by taking each element \(l_{j} \in S_{k, n+2}\) repeated \(p_{l, j}\) times. Expressing the previous sum in terms of the sequences \((u_{i})_{1 \leq i \leq \gamma}\) instead gives the requested expression.
\end{proof}

We may now rewrite the two intersection numbers as polynomials in \(d\):
	\begin{align*}
		(k!)^{n} A^{N_{k}} & =
			\sum_{l = 0}^{n} 2^{l} \binom{N_{k}}{l} \left( \sum_{0 \leq \alpha \leq n - l} \Lambda_{\alpha, n-l} d^{\alpha} \right) \cdot \left(\int_{X} h^{n}\right) \\
		& = \sum_{l = 0}^{n} 2^{l} \binom{N_{k}}{l} \left( \sum_{\alpha + \gamma = n - l} (-1)^{\gamma} B_{\gamma} C_{\alpha} d^{\alpha}\right) \cdot \left(\int_{X} h^{n}\right) \\
		& = d \sum_{\alpha = 0}^{n} \left[ 
		\sum_{l = 0}^{n- \alpha} \binom{N_{k}}{l} 2^{l} (-1)^{n-\alpha -l} B_{n-\alpha -l} 
		\right] 
		C_{\alpha}
		d^{\alpha}
		\quad
		\left( \text{since} \int_{X} h^{n} = d \right)
	\end{align*}

	Similarly, one has
	\begin{align*}
		(k!)^{n} A^{N_{k}-1} \cdot B & = \sum_{l= 1}^{n} \binom{N_{k} - 1}{l-1} 2^{l-1}(2 + \epsilon) \left( \sum_{0 \leq \alpha \leq n - l} \Lambda_{\alpha, n-l} d^{\alpha} \right) h^{l} \\
		& = \sum_{l=1}^{n} \binom{N_{k}-1}{l-1} 2^{l-1}(2 + \epsilon) \left[ \sum_{\alpha + \gamma = n-l} (-1)^{\gamma} B_{\gamma} C_{\alpha} d^{\alpha}\right] \cdot (\int_{X} h^{n}) \\
		& = d \sum_{\alpha=0}^{n} \left[\sum_{l=1}^{n-\alpha} \binom{N_{k}-1}{l-1} 2^{l-1}(2+\epsilon) (-1)^{n-l-\alpha} B_{n-\alpha-l} \right] C_{\alpha} d^{\alpha}
	\end{align*}

This shows that:
\[
	A^{N_{k}} - N_{k} A^{N_{k} -1} \cdot B
	=
	d \sum_{\alpha=0}^{n} Q_{\alpha}(n, \epsilon) d^{\alpha}
\]
with
\begin{align*}
	Q_{\alpha}(n, \epsilon)
	=
	C_{\alpha} \left[
	\sum_{l=0}^{n-\alpha} \binom{N_{k}}{l} 2^{l} (-1)^{n-\alpha-l} B_{n-\alpha-l}
	-
	\sum_{l=1}^{n - \alpha} N_{k} \binom{N_{k} -1}{l-1} 2^{l-1} (2 + \epsilon) (-1)^{n-\alpha-l} B_{n-\alpha-l} \right] \\
\end{align*}

Let us rewrite this by singling out the term \(l=0\) and merging the two sums using the formula \(a \binom{b}{a} = b \binom{b-1}{a-1}\):
\begin{align*}
	Q_{\alpha}(n, \epsilon)
	=
	(-1)^{n-\alpha} C_{\alpha}  \left[
	B_{n-\alpha} +
	\sum_{l=1}^{n-\alpha} ( 2 - (2 + \epsilon)l) \binom{N_{k}}{l} (-1)^{l} 2^{l-1} B_{n-\alpha-l} 
	\right]
\end{align*}

\noindent
{\em Step 4. Bounds on the coefficients.} We now fix \(k = n\). To obtain the bound \eqref{eq:fujiwarabound}, we are simply going to drop all the signs in the expression of \(Q_{\alpha}\) to define
\[
	R_{\alpha}
	=
	C_{\alpha}  \left[
	B_{n-\alpha} +
	\sum_{l=1}^{n-\alpha} ( 2 + (2 + \epsilon)l) \binom{N_{k}}{l} 2^{l-1} B_{n-\alpha-l} 
	\right]
\]
and then use the very coarse inequality \(| Q_{\alpha} | \leq R_{\alpha}\). We may rewrite \(R_{\alpha}\) as
\begin{equation} \label{eq:exprRa}
	R_{\alpha}
	=
	C_{\alpha}
	\left[
	B_{n-\alpha}
	+ \sum_{l=1}^{n-\alpha} D_{l} B_{n-\alpha - l}
	\right]
\end{equation}
with
\[
	D_{l}
	=
	( 2 + (2 + \epsilon)l) \binom{N_{k}}{l} 2^{l-1}
\]

The main estimates on the \(R_{\alpha}\) will come from \eqref{eq:exprRa} and the following inequalities:

\begin{lemma} We have the following upper bounds: 
	\begin{enumerate}
		\item \(B_{\alpha+1} \leq 2 n^{2} \, B_{\alpha}\)
		\item \(C_{\alpha} \leq \frac{3}{2} n^{2}\,  C_{\alpha + 1}\) if \(\alpha + 1 \leq n\)
		\item 	\(D_{\alpha +1} \leq 9 n^{2} D_{\alpha}\) for \(\alpha \geq 1\).
	\end{enumerate}
\end{lemma}

\begin{proof}[Proof of the lemma]
	{\em (1).} One has
	\begin{align*}
		B_{\alpha + 1} & =  \sum_{\{u_{1} \leq \dotsc \leq u_{\alpha} \leq u_{\alpha +1}\} \subset S_{n, n+2}} \frac{1}{u_{1} \dotsc u_{\alpha}u_{\alpha +1}} \\
			       & \leq \sum_{\{u_{1} \leq \dotsc \leq u_{\alpha}\} \subset S_{n, n+2}} \sum_{u_{\alpha+1} \in S_{n, n+2}} \frac{1}{u_{1} \dotsc u_{\gamma}} \frac{1}{u_{\alpha+1}} \\ 
			       & \leq n(n+2)  \sum_{\{u_{1} \leq \dotsc \leq u_{\alpha}\} \subset S_{n, n+2}} \frac{1}{u_{1} \dotsc u_{\gamma}} \\
			       & \leq n(n+2) B_{\alpha} \leq 2 n^{2} B_{\alpha}.
	\end{align*}
	To obtain the last inequality, recall that \(n \geq 2\).
	\medskip

	{\em (2).} One has, if \(\alpha + 1 \leq n\):
	\begin{align*}
		C_{\alpha + 1} & = \sum_{1 \leq l_{1} < \dotsc < l_{\alpha} < l_{\alpha +1} \leq n} \frac{1}{l_{1} \dotsc l_{\alpha} l_{\alpha + 1}} \\
				&  = \frac{1}{\alpha + 1} \sum_{S \subset \llbracket 1, n\rrbracket, |S|= \alpha}\; \sum_{l \not\in S} \frac{1}{\mathrm{prod}(S)} \frac{1}{l} \\
				& \geq \frac{1}{n+1} \sum_{S \subset \llbracket 1, n\rrbracket, |S|= \alpha}\; \sum_{l \not\in S} \frac{1}{\mathrm{prod}(S)} \frac{1}{n} \\
				& = \frac{n - \alpha}{n(n+1)} C_{\alpha} \geq \frac{1}{n(n+1)} C_{\alpha} \geq \frac{2}{3n^{2}} C_{\alpha}.
	\end{align*}
	\medskip

	{\em (3).} One has
	\begin{align*}
		D_{\alpha + 1} & = ( 2 + (2 + \epsilon)(\alpha + 1)) \binom{N_{n}}{\alpha + 1} 2^{\alpha}
		\\
		& \leq 6 N_{n} ( 2 + (2 + \epsilon)\alpha) \binom{N_{n}}{\alpha} 2^{\alpha - 1} = 6 N_{n} D_{\alpha}
	\end{align*}
	where we used \(\binom{N_{n}}{\alpha+1} = \frac{N_{n}-\alpha}{\alpha+1} \binom{N_{n}}{\alpha} \leq N_{n} \binom{N_{n}}{\alpha}\) and
	\[
		\frac{2 + (2 + \epsilon) (\alpha + 1)}{2 + (2 + \epsilon) \alpha}
		\leq
		\frac{2 + (2 + \epsilon) (\alpha + 1)}{(2 + \epsilon) \alpha}
		\leq
		1 + \frac{\alpha + 1}{\alpha} \leq 3.
	\]
	Note that \(N_{n} = n + n^{2} - 1 \leq \frac{3}{2}n^{2}\), so finally we find
	\[
		D_{\alpha + 1}  \leq 9 n^{2} D_{\alpha}.
	\]
\end{proof}

Finally, we obtain:
\begin{lemma}
	There exists a constant \(D\) independent of \(n, d, \epsilon\), such that
	\begin{enumerate}
		\item \(R_{n-1} \leq 9 n^{4}\, (1 + \frac{\epsilon}{4}) R_{n}\)
					\item for all \(1 \leq \alpha \leq n-1\), we have
			\[
				R_{\alpha-1} \leq \frac{27}{2} n^{4} R_{\alpha}.
			\]
	\end{enumerate}
\end{lemma}
\begin{proof}
\noindent
{\em (1)} One has
\[
	R_{n} = C_{n} B_{0} = C_{n} > 0
\]
and
\begin{align*}
	R_{n-1} & = C_{n-1} (B_{1} + D_{1} B_{0}) \\
		& = C_{n-1} (B_{1} + (4 + \epsilon)N_{n}) \\
		& \leq \frac{3}{2} n^{2}\, C_{n} \cdot (2n^{2} + (4 + \epsilon)\frac{3}{2} n^{2}) 
		=  9 n^{4}(1 + \frac{\epsilon}{4}) R_{n}.
\end{align*}

\noindent
{\em (2)} One has
\begin{align*}
	R_{\alpha-1} & = C_{\alpha-1}
	\left[
	B_{n-\alpha + 1}
	+ \sum_{l=1}^{n-\alpha+1} D_{l} B_{n-\alpha + 1 - l}
	\right] \\
	& = C_{\alpha-1}
	\left[
	B_{n-\alpha + 1}
	+ \sum_{l=1}^{n-\alpha} D_{l} B_{n-\alpha + 1 - l}
	+ D_{n-\alpha+1}
	\right] \\
	& \leq \frac{3}{2} n^{2}\, C_{\alpha}
	\left[
	2 n^{2}\, B_{n-\alpha}
	+ 2 n^{2}\, \sum_{l=1}^{n-\alpha} D_{l} B_{n-\alpha - l}
	+ 9 n^{2} D_{n-\alpha} 
	\right] \\
	& = \frac{27}{2} n^{4} C_{\alpha} 
	\left[
		\frac{2}{9} B_{n-\alpha}
	+ \frac{2}{9} \sum_{l=1}^{n-\alpha-1} D_{l} B_{n-\alpha - l}
	+  D_{n-\alpha} 
	\right] \\
	& \leq
	\frac{27}{2} n^{4} R_{\alpha}.
\end{align*}
Again, one gets the result.
\end{proof}

This now proves item {\em (2)} of Proposition~\ref{prop:mainresult}: the previous lemma shows inductively that
\[
	R_{\alpha} \leq D_{\epsilon}^{n-j} n^{4(n-j)} R_{n},
\]
with \(D_{\epsilon} = \max(9(1 + \frac{\epsilon}{4}), \frac{27}{4})\).
To obtain the result, just use the fact that \(Q_{n} = C_{n} = R_{n} >0\) and the inequalities \(|Q_{\alpha}| \leq R_{\alpha}\).
\medskip

\noindent
{\em Step 5. Fujiwara bound.} Recall the following elementary result:

\begin{lemma} \label{lem:fujiwara}
	Let \(Q(t) = a_{n} t^{n} + a_{n-1} t^{n-1} + \cdots + a_{0}\) be a polynomial with real coefficients such that \(a_{n} > 0\). Assume there exists a constant \(M > 0\) such that
	\begin{equation} \label{eq:fujiwara}
		|a_{n-j}| \leq M^{j} a_{n}
	\end{equation}
	for all \(j > 0\). Then for \(t > 2M\), one has \(Q(t) >0\).
\end{lemma}

In our situation, the polynomial \(Q\) is given by
\[
	Q(t) = \sum_{j=1}^{n} Q_{j}(n, \epsilon) t^{j}, 
\]
(see \eqref{eq:polynomial}), where \(n, \epsilon\) are seen as constants. Then by item {\em (2)} of Proposition~\ref{prop:mainresult}, one has the bound \eqref{eq:fujiwara} with
\[
	M = D_{\epsilon} n^{4}.
\]
This gives the result.
\vfill

\section{Annex. Intersection theory on weighted projective bundles} \label{sec:annex}

The purpose of this annex is mainly to give an elementary proof of the Whitney formula~\eqref{eq:whitney}, as an equality between numerical classes (which is enough to perform the computations presented in this article). The route we follow will be closely related to the original computations of Green-Griffiths \cite{GG80} (i.e. estimating Euler characteristics by integrals of polynomials on adequate simplexes) -- we also used this kind of computations in the work \cite{Cad19}.
\medskip

Let \(X\) be a complex projective manifold, endowed with a weighted vector bundle \(\mathbf{E} := E_{1}^{(a_{1})} \oplus \dotsc \oplus E_{r}^{(a_{r})}\) (the reader can refer to Section~\ref{sec:weighted} for the terminology). Then the Whitney formula that was proved in the author's thesis can be stated as follows:

\begin{proposition} [{\cite[Proposition~3.2.11]{cad_thesis}}]
We have the following equality of endomorphisms of \((A_{\ast} X)_{\mathbb{Q}}\):
\begin{equation} \label{eq:whitney}
	s_{\bullet}(E_{1}^{(a_{1})} \oplus \dotsc \oplus E_{s}^{(a_{s})})
	=
	\frac{\mathrm{gcd}(a_{1}, \dotsc, a_{s})}{a_{1} \dotsc a_{s}}
	\prod_{j} s_{\bullet}(E_{j}^{(a_{j})})
\end{equation}
where
	\(
	s_{\bullet}(E^{(a)})
	=
	\frac{1}{a^{\mathrm{rk} E - 1}} \sum_{l} \frac{s_{l}(E)}{a^{l}}.
	\)
	for a single vector bundle \(E\) and any integer \(a > 0\).
	\medskip
\end{proposition}

The proof presented in \cite{cad_thesis} essentially copies Fulton's presentation in \cite{ful98}. As announced above, we will show that this formula holds at least after quotienting by the numerical equivalence relation -- the idea will be to see the top Segre class as the leading coefficient in the asymptotic expansion of the Euler characteristic of the symmetric products of \(\mathbf{E}^{\ast}\). As the reader will see, this type of computation is very close in spirit to the ones of \cite{GG80}. We will need several combinatorial lemmas, all gathered in Section~\ref{sec:combinatorial} below.

\begin{proof}[Proof of the identity as endomorphisms of \((A_{\ast} X)_{\mathbb{Q}, \mathrm{num}}\).] We will prove that for any integer \(k \in \llbracket 1, n \rrbracket\) and any cycle \(\alpha \in A_{k} X\), one has for \(1 \leq l \leq k\):
	\[
	s_{l}(E_{1}^{(a_{1})} \oplus \dotsc \oplus E_{s}^{(a_{s})}) \cap \alpha
	\equiv_{\mathrm{num}}	
	\left( \frac{\mathrm{gcd}(a_{1}, \dotsc, a_{s})}{a_{1} \dotsc a_{s}}
	\prod_{j} s_{\bullet}(E_{j}^{(a_{j})}) \right)_{l} \cap \alpha
	\]

	{\em Step 1. Reduction steps.} The previous formula means that intersecting both sides with any \(\beta \in A_{k-l}(X)\) must give the same intersection number; we see right away that replacing \(\alpha\) by \(\alpha \cap \beta\), allows to reduce to the case \(l = k\), where one has to show equality between the two sides of the equation, which are seen as elements of \(\mathbb{Z}\). Also, breaking \(\alpha\) into its irreducible components, we see that we may finally assume \(\alpha = [X]\) and \(j = n\). By the classical splitting principle (see \cite[p.51]{ful98}), one may also assume that each \(E_{i}\) is split as a sum of line bundles.
	\medskip

	{\em Step 2. Proof of the formula in the split case.} Let us then assume that each \(E_{i} = L_{i, 1} \oplus \dotsc \oplus L_{i, r_{i}}\) is split as a direct sum of line bundles for all \(1 \leq i \leq s\), and let \(\alpha_{i, j} = c_{1} (L_{i, j})\) denote the corresponding Chern roots. We will rather prove the dual formula for \(s_{n}(\mathbf{E}^{\ast})\) to make the signs easier to track. 

	Let \(m_{0} := \gcd(a_{1}, \dotsc, a_{r})\). Then, applying the asymptotic Riemann-Roch theorem, one has that 
	\begin{align*}
		\int_{X} s_{n}(\mathbf{E}^{\ast}) & = \frac{1}{m_{0}^{n+r-1}}\int_{\mathbb{P}(\mathbf{E})} c_{1} \mathcal{O}(m_{0})^{n+r-1} \quad \text{(by definition of } s_{n}(\mathbf{E}^{\ast}))\\
					& = \lim_{\stackrel{m \longrightarrow +\infty}{m_{0} | m}} 
					\frac{\chi(\mathbb{P}(\mathbf{E}), \mathcal{O}(m))}{m^{n+r-1}/(n+r-1)!} 
	\end{align*}
	By the Leray spectral sequence, one also has
	\(
		\chi(\mathbb{P}(\mathbf{E}), \mathcal{O}(m))) = \chi(X, \pi_{\ast} \mathcal{O}(m)) = \chi(X, S^{m}\mathbf{E}). 
	\)
	But then, since each \(E_{i}\) is split, Proposition~\ref{prop:symmetricsum} below implies that
	\[
		\int_{X} s_{n}(\mathbf{E}^{\ast}) = 
		\frac{m_{0}}{a_{1}^{r_{1}} \dotsc a_{s}^{r_{s}}}
		\int_{X} \left( \prod_{i=1}^{s} \prod_{j=1}^{r_{i}} \sum_{p=0}^{n} \left(\frac{\alpha_{i, j}}{a_{i}}\right)^{p} \right)_{n}.
	\]
	Now, we see that for all \(i = 1, \dotsc, s\), the expression  \( \frac{1}{a_{i}^{r_{i} - 1}} \prod_{j=1}^{r_{i}} \sum_{p=0}^{n} \left(\frac{\alpha_{i, j}}{a_{i}}\right)^{p}\) identifies with the formula for \(s_{\bullet}((E_{i}^{\ast})^{(a_{i})})\) that is given in the statement of the proposition. Thus, one obtains
	\[
		\int_{X} s_{n}(\mathbf{E}^{\ast})
		= \frac{m_{0}}{a_{1} \dotsc a_{r}} 
		\int_{X} 
		\left( \prod_{i=1}^{s} s_{\bullet}((E_{i}^{\ast})^{(a_{i})}) \right)_{n},
	\]
	which gives the result.
\end{proof}

\subsection{Combinatorial lemmas} \label{sec:combinatorial}

The purpose of the following discussion is to prove Proposition~\ref{prop:symmetricsum}, that was used in the proof of the Whitney formula. We need several combinatorial lemmas, closely related to several computations that appeared in \cite{Cad19}.

\noindent
{\bf Notation.} 
\begin{enumerate}
	\item For all $n$, we let $\mathrm{vol}_{n}$ denote the $n$-dimensional euclidian volume measure.
	\item Let $m \in \mathbb N$. A \emph{$m$-dimensional simplex} $\Delta$ is a metric space isomorphic to the convex envelop of $m+1$ points in $\mathbb R^m$, such that each $p$ of them generate an affine $(p-1)$-space. We will sometimes write $\Delta \; \accentset{\circ}{\subset} \; \mathbb R^m$ to emphasize the fact that $\Delta$ has non-empty interior in $\mathbb R^m$ (or equivalently, that $\dim \Delta = m$) and to oppose this situation to the case of a $(m-1)$-dimensional simplex included in $\mathbb R^m$. 

	\item Recall that for any $m$-dimensional simplex $\Delta \accentset{\circ}{\subset} \mathbb R^m$, the \emph{uniform probability measure} of $\Delta$ is the measure $d \mathbf P_{\Delta} = \frac{1}{\mathrm{vol}_m(\Delta)} d \mathrm{vol}_m$. 

Since this measure on $\Delta$ is the unique probability measure which is the restriction of a translation invariant measure on $\mathbb R^m$, we see that if $\Delta_1, \Delta_2 \subseteq \mathbb R^m$ are $m$-dimensional simplexes, and if $\Psi \in \mathrm{GL}(\mathbb R^m)$ is such that $\Delta_2 = \Psi(\Delta_1)$, then $\Psi$ sends the uniform measure of $\Delta_1$ on the uniform measure of $\Delta_2$.

\end{enumerate}

\subsection{Lattices and volumes of fundamental domains}

Let $a_1, ..., a_r \in \mathbb N$. Let 
\[
	H = \{ (t_1, ..., t_r) \in \mathbb Z^r \; | \; \sum_i a_i t_i = 0\}.
\] Then $H \subseteq \mathbb Z^r$ is a \emph{primitive} sublattice, meaning that $\quotientd{\mathbb Z^r}{H}$ is torsion-free. Hence, by the adapted basis theorem, there exists a basis $(f_1, ..., f_r)$ of $\mathbb Z^r$ such that $(f_1, ..., f_{r-1})$ is in turn a basis for $H$. Let $C_H = \sum_{1 \leq i \leq r-1} [0, 1] \cdot f_i$ denote the associated fundamental domain of $H$. Note that all fundamental domains are image of one another by an element of \(\mathrm{SL}(H)\) so they all have the same \((n-1)\)-volume.

\begin{lemma} \label{lemlattice1} Any fundamental domain of $H$ has volume $\mathrm{vol}_{r-1} \left( C_H \right) = \frac{\sqrt{ \sum_{1 \leq i \leq r } a_i^2}}{\gcd(a_1, ..., a_r)}$.
\end{lemma}
\begin{proof}
	The lattice $H$ and the proposed formula for the volume do not change if we replace $a_i$ by $\frac{a_i}{\gcd(a_1, ..., a_r)}$, hence we can suppose that $\gcd(a_1, ..., a_r) = 1$. In this case, there exist $u_1, ..., u_r \in \mathbb Z$ such that $\sum_i a_i u_i = 1$. Replacing \(f_{r}\) by the vector \((u_{1}, \dotsc, u_{r})\) in \((f_{1}, \dotsc, f_{r})\) still gives a basis of \(\mathbb{Z}^{r}\), so we can assume that $f_r = (u_1, ..., u_r)$.

Since $(f_1, ..., f_{r})$ is a basis of $\mathbb Z^r$, we have $\mathrm{vol}_r(\sum_{1 \leq i \leq r} [0, 1] \cdot f_i) = 1$. Moreover, 
\begin{align*}
\mathrm{vol}_r(\sum_{1 \leq i \leq r} [0, 1] \cdot f_i) & = \mathrm{vol}_{r-1}(\sum_{1 \leq i \leq r - 1} [0, 1] \cdot f_i) \, \cdot \, \norm{\pi_{H^\perp} (f_r) }_{\mathrm{eucl}} \\
	& = \mathrm{vol}_{r-1}(C_H) \, \cdot\, \norm{\pi_{H^\perp} (f_r) }_{\mathrm{eucl}}
\end{align*}
	where $\pi_{H^\perp} (f_r)$ is the orthogonal projection of $f_r$ on $H^\perp$, and $\norm{ \cdot }_{\mathrm{eucl}}$ is the euclidian norm. We obtain
	\[
		\mathrm{vol}_{r-1}(C_{H}) = \frac{1}{\norm{\pi_{H^{\perp}}(f_{r})}}.
	\]

	Let us now compute \(\norm{\pi_{H^{\perp}}(f_{r})}\). Since $H^\perp = \mathbb R \cdot (a_1, \dotsc, a_r)$ by definition of $H$, one can write
	\[
		f_{r} = (u_{1}, \dotsc, u_{r}) = \lambda \cdot (a_{1}, \dotsc, a_{r}) + (b_{1}, \dotsc, b_{r}),
	\]
	with \(\lambda \in \mathbb{R}\) and \((b_{1}, \dotsc, b_{r}) \in H_{\mathbb{R}}\). Thus, one has
	\[
		0 = \sum_{j} a_{j} b_{j} = \sum_{j} a_{j} u_{j} - \lambda \sum_{j} a_{j}^{2}.
	\]
	 Since \(\sum_{j} a_{j} u_{j} = 1\), this gives \(\lambda = \frac{1}{\sum_{j} a_{j}^{2}}\) and thus \(\norm{\pi_{H^\perp} (f_r) }_{\mathrm{eucl}}^{2} = \lambda^{2} \sum_{j} a_{j}^{2} = \frac{1}{\sum_{j} a_{j}^{2}}\). We obtain the result.
\end{proof}

\begin{lemma} \label{lem:lattice2} Let $\underline{a} = (a_1, ..., a_r) \in \mathbb N^r$, and let $\Delta_{\underline{a}} = \{ (t_i) \in \mathbb R_+^r \; | \; \sum_i a_i t_i = 1 \}$. Then the volume of $\Delta_{\underline{a}}$ is $\mathrm{vol}_{r-1} (\Delta_{\underline{a}}) = \frac{1}{(r-1)!} \frac{\gcd(a_1, ..., a_r)}{a_1 ... a_r} \mathrm{vol}_{r-1}(C_H)$;
\end{lemma}
\begin{proof}
	By Lemma \ref{lemlattice1}, it suffices to show that $\mathrm{vol}_{r-1}(\Delta_{\underline{a}}) = \frac{1}{(r-1)!} \frac{\sqrt{\sum_{1 \leq i \leq r} a_i^2}}{a_1 ... a_r}$. To perform this computation, we can for example use the parametrization of $\Delta_{\underline{a}}$ given by $\psi : t \in \Delta   \longmapsto (\frac{1}{a_1} t_1, ..., \frac{1}{a_{r-1}} t_{r-1}, \frac{1}{a_r} (1 - \sum_{1 \leq i \leq r - 1} t_i))$, where $\Delta = \{ (t_i) \in [0,1]^{r-1} \; | \; \sum_i t_i \leq 1 \}$ is the standard $(r-1)$-dimensional simplex in $\mathbb R^{r-1}$. We have then $\psi^\ast( d\mathrm{vol}_{r-1} ) = \sqrt{\det G} \, d\mathrm{vol}_{r-1}$, where $G= (\left< \psi_\ast(e_i), \psi_\ast(e_j)\right>)_{i,j}$ is the Gram matrix of the vectors $\psi_\ast (e_i)$ ($(e_i)_i$ being the canonical basis of $\mathbb R^{r-1}$). A simple computation shows that $\det G = \frac{1}{\prod_i a_i^2} \sum_{i} a_i^2$. Thus, we have $\mathrm{vol}_{r-1}(\Delta_{\underline{a}}) = \frac{\sqrt{\sum_i a_i^2}}{\prod_i a_i} \mathrm{vol}_{r-1}(\Delta)$. To conclude, it suffices to compute $\mathrm{vol}_{r-1}(\Delta) = \frac{1}{(r-1)!}$, which is easy.
\end{proof}

\subsection{Some integrals} We are now going to compute the integral of some monomial functions on simplexes with respect to the uniform probability measure. The goal is to prove the following.

\begin{lemma} \label{lem:integral} Let \(\underline{a} = (a_{1}, \dotsc, a_{r}) \in \mathbb{N}^{r}\), and let \(\Delta_{\underline{a}} = \{ (t_{i}) \in \mathbb{R}_{+}^{r} \; | \; \sum_{i} a_{i} t_{i} = 1 \}\). Let \(p_{1}, \dotsc, p_{r} \in \mathbb{N}\). Then
	\[
		\int_{\Delta_{\underline{a}}} t_{1}^{p_{1}} \dotsc t_{r}^{p_{r}} 
		d \mathbf{P}_{\Delta_{\underline{a}}}(t)
		=
		\frac{(r-1)!\, p_{1}! \dotsc p_{r}!}{(p_{1} + p_{2} + \dotsc + p_{r} + r - 1)!}
		\frac{1}{a_{1}^{p_{1}} \dotsc a_{r}^{p_{r}}}.
	\]
\end{lemma}

First, let us remark that we can easily get back to the case where \(\underline{a}\) is equal to \(\underline{1} := (1, \dotsc, 1)\). Indeed, letting \(\Psi(t_{1}, \dotsc, t_{r}) = (a_{1} t_{1}, \dotsc, a_{r} t_{r})\), one has \(\Psi_{\ast} d\mathbf{P}_{\Delta_{\underline{1}}} = d \mathbf{P}_{\Delta_{\underline{a}}}\), so that
\begin{align*}
\int_{\Delta_{\underline{a}}} t_{1}^{p_{1}} \dotsc t_{r}^{p_{r}} 
		d \mathbf{P}_{\Delta_{\underline{a}}}(t)
	& =
	\int_{\Delta_{\underline{1}}} 
	\left(\frac{t_{1}}{a_{1}}\right)^{a_{1}}
	\dotsc
	\left(\frac{t_{r}}{a_{r}}\right)^{a_{r}}
	d \mathbf{P}_{\Delta_{\underline{1}}}(t) \\
	& =
	\frac{1}{a_{1}^{p_{1}} \dotsc a_{r}^{p_{r}}} \int_{\Delta_{\underline{1}}} t_{1}^{p_{1}} \dotsc t_{r}^{p_{r}} d \mathbf{P}_{\Delta_{\underline{1}}}(t)
\end{align*}

For any \(r \in \mathbb{N}\) and \(p_{1}, \dotsc, p_{r} \in \mathbb{N}\), let \(C_{p_{1}, \dotsc, p_{r}} := \int_{\Delta_{\underline{1}}} t_{1}^{p_{1}} \dotsc t_{r}^{p_{r}} d \mathbf{P}_{\Delta_{\underline{1}}}(t)\). By the remark above, the proof of Lemma~\ref{lem:integral} will then be complete with the following result.
\begin{lemma} We have \(C_{p_{1}, \dotsc, p_{r}} = 
	\frac{(r-1)!\, p_{1}! \dotsc p_{r}!}{(p_{1} + p_{2} + \dotsc + p_{r} + r - 1)!}.
	\)
\end{lemma}
\begin{proof}
	By induction on \(r\). Letting \(\Delta_{r} := \Delta_{(1, \dotsc, 1)}\) to simplify the notation (\(1\) is repeated \(r\) times), one has
	\begin{align*}
		C_{p_{1}, \dotsc, p_{r}} & = 
		\int_{t_{1} + \dotsc + t_{r} = 1} 
		t_{1}^{p_{1}} \dotsc t_{r}^{p_{r}} 
		\frac{dt_{1} \wedge \dotsc \wedge dt_{r-1}}{\mathrm{vol}_{dt_{1} \wedge \dotsc \wedge dt_{r-1}}(\Delta_{r})} \\
		& = \int_{t_{1} + s = 1} dt_{1}\, t_{1}^{p_{1}}\, 
		\frac{\mathrm{vol}_{dt_{2} \wedge \dotsc \wedge dt_{r-1}}(s \Delta_{r-1})}{\mathrm{vol}_{dt_{1} \wedge \dotsc \wedge dt_{r-1}}(\Delta_{r})}  
		\int_{t_{2} + \dotsc + t_{r} = s} t_{2}^{p_{2}} \dotsc t_{r}^{p_{r}} 
		\frac{dt_{2} \wedge \dotsc \wedge dt_{r-1}}{\mathrm{vol}_{dt_{2} \wedge \dotsc \wedge dt_{r-1}}(s \Delta_{r-1})} \\
		& = \int_{t_{1} + s = 1} dt_{1}\, t_{1}^{p_{1}}\, 
		\frac{s^{r-2}/(r-2)!}{1/(r-1)!}  
		\cdot 
		s^{p_{2} + \dotsc + p_{r}} \int_{t_{2} + \dotsc + t_{r} = 1} t_{2}^{p_{2}} \dotsc t_{r}^{p_{r}} 
		\frac{dt_{2} \wedge \dotsc \wedge dt_{r-1}}{\mathrm{vol}_{dt_{2} \wedge \dotsc \wedge dt_{r-1}}(\Delta_{r-1})} \\
		& = (r-1) \int_{t_{1} + s = 1} dt_{1} \, t_{1}^{p_{1}}\, s^{r - 2 + p_{2} + \dotsc + p_{r}} C_{p_{2}, \dotsc, p_{r}} 
		= (r-1) \frac{p_{1}!(p_{2} + \dotsc + p_{r} + r-2)!}{(p_{1} + p_{2} + \dotsc + p_{r} + r - 1)!} C_{p_{2}, \dotsc, p_{r}}
	\end{align*}
	where at the last line, we used Lemma~\ref{lem:Bfunction} below. This permits to prove the formula by induction.
\end{proof}

\begin{lemma} \label{lem:Bfunction} 
	Let \(a, b \in \mathbb{N}\). One has
	\[
		\int_{0}^{1} t^{a} (1 - t)^{b} dt = \frac{a! b!}{(a + b + 1)!}.
	\]
\end{lemma}
\begin{proof} This comes from the Beta function identity \(\int_{0}^{1} t^{a} (1 - t)^{b} dt = \frac{\Gamma(a+1) \Gamma(b+1)}{\Gamma(a+b+2)}\). 
\end{proof}

\subsection{Riemann integrals and asymptotic estimates}

Using the previous results, we can now give the following asymptotic estimates that will prove useful to compute the asymptotics of Euler characteristics.

\begin{lemma} \label{lem:asymptoticsriemann} Fix integers \(a_{1}, \dotsc, a_{r} \in \mathbb{N}\), and \(p_{1}, \dotsc, p_{r} \in \mathbb{N}\). We have, for \(m \longrightarrow + \infty\) divisible by \(\gcd(a_{1}, \dotsc, a_{r})\):
	\[
		\sum_{a_{1} l_{1} +  \dotsc + a_{r} l_{r} = m}
		\frac{l_{1}^{p_{1}}}{p_{1}!} \dotsc \frac{l_{r}^{p_{r}}}{p_{r}!} = 
		\frac{\gcd(a_{1}, \dotsc, a_{r})}{a_{1}^{p_{1} + 1} \dotsc a_{r}^{p_{r} + 1}}
		\frac{m^{p_{1} + \dotsc + p_{r} + r - 1}}{(p_{1} + p_{2} + \dotsc + p_{r} + r - 1)!}
		+
		o(m^{p_{1} + \dotsc + p_{r} + r - 1}).
	\]
\end{lemma}
\begin{proof} Let \(H_{m}\) be the set of \((l_{1}, \dotsc, l_{r}) \in \mathbb{Z}^{r}\) such that \(\sum_{j} l_{j} a_{j} = m\). It is non-empty if and only if \(\gcd(a_{1}, \dotsc, a_{r}) | m\), and if such is the case, it is then a translate of the lattice \(H = \{ (l_{1}, \dotsc, l_{r}) \in \mathbb{Z}^{r} \; | \; \sum_{j} a_{j} l_{j} = 0\}\). Let \(C_{H}\) denote a fundamental domain for \(H\).
	\medskip

	As \((l_{1}, \dotsc, l_{r})\) varies in \(H_{m} \cap \mathbb{N}^{r}\), the element \((\frac{l_{1}}{m}, \dotsc, \frac{l_{r}}{m})\) varies in \(\Delta_{\underline{a}}\), running in a lattice with cells isometric to \(\frac{1}{m}C_{H}\). Thus, one can use a Riemann sum to obtain
	\begin{align*}
		\mathrm{vol}_{r-1}(\frac{1}{m} C_{H}) \cdot
		\sum_{a_{1} l_{1} +  \dotsc + a_{r} l_{r} = m}
		\left(\frac{l_{1}}{m}\right)^{p_{1}} 
		\dotsc 
		\left(\frac{l_{r}}{m}\right)^{p_{r}}
		& \underset{m \longrightarrow + \infty}{\longrightarrow}
		\int_{\Delta_{\underline{a}}} t_{1}^{p_{1}} \dotsc t_{r}^{p_{r}} d \mathrm{vol}_{r-1}(t).
		\\
		& = \mathrm{vol}_{r-1}(\Delta_{\underline{a}})
		\int_{\Delta_{\underline{a}}} t_{1}^{p_{1}} \dotsc t_{r}^{p_{r}} d \mathbf{P}_{\Delta_{\underline{a}}}(t). 
	\end{align*}
	Thus we deduce
	\[
		\frac{1}{m^{p_{1} + \dotsc + p_{r} + r - 1}}	
		\sum_{a_{1} l_{1} +  \dotsc + a_{r} l_{r} = m}
		l_{1}^{p_{1}} \dotsc l_{r}^{p_{r}}
		\longrightarrow
		\frac{\mathrm{vol}_{r-1}(\Delta_{\underline{a}})}{\mathrm{vol}_{r-1}(C_{H})}
		\int_{\Delta_{\underline{a}}} t_{1}^{p_{1}} \dotsc t_{r}^{p_{r}} d \mathbf{P}_{\Delta_{\underline{a}}}(t). 
	\]
	The right hand side can be computed using Lemmas~\ref{lem:lattice2} and \ref{lem:integral}. This gives the result.
\end{proof}

We will need another version of that lemma for our application to the asymptotic Riemann-Roch theorem.

\begin{lemma} \label{lem:asymptfraction} Let \(n, r \in \mathbb{N}\) be two integers. Let \(\alpha_{1}, \dotsc, \alpha_{r}\) be indeterminates over \(\mathbb{C}\). Fix integers \(a_{1}, \dotsc, a_{r} \in \mathbb{N}\), and \(p_{1}, \dotsc, p_{r} \in \mathbb{N}\). We have, for \(m \longrightarrow + \infty\) divisible by \(\gcd(a_{1}, \dotsc, a_{r})\):
	\[
		\sum_{a_{1} l_{1} + \dotsc + a_{r} l_{r} = m}
		\frac{(\alpha_{1} l_{1} + \dotsc + \alpha_{r} l_{r})^{n}}{n!}
		=
		\frac{\gcd(a_{1}, \dotsc, a_{r})}{a_{1} \dotsc a_{r}} 
		\left[
			\sum_{p_{1} + \dotsc + p_{r} = n} \left(\frac{\alpha_{1}}{a_{1}}\right)^{p_{1}} \dotsc \left(\frac{\alpha_{1}}{a_{r}}\right)^{p_{r}} 
		\right]
		\frac{m^{n + r -1}}{(n+ r - 1)!}
		+ o(m^{n+r-1}).
	\]
	where \(o(m^{n+r-1})\) means a homogeneous polynomial of degree \(n\) in \(\alpha_{1}, \dotsc, \alpha_{n}\), all of whose coefficients are negligeable compared to \(m^{n+ r -1}\).
\end{lemma}
\begin{proof}
	We expand the sum, using the Newton identity, and we apply Lemma~\ref{lem:asymptoticsriemann}: 
	\begin{align*}
		\sum_{a_{1} l_{1} + \dotsc + a_{r} l_{r} = m}
		\frac{(\alpha_{1} l_{1} + \dotsc + \alpha_{r} l_{r})^{n}}{n!} & =
		\sum_{a_{1} l_{1} + \dotsc + a_{r} l_{r} = m}
		\;
		\sum_{p_{1} + \dotsc + p_{r} = n} \binom{n}{p_{1}, \dotsc, p_{r}}
		l_{1}^{p_{1}} \dotsc l_{r}^{p_{r}} \alpha_{1}^{p_{1}} \dotsc \alpha_{r}^{p_{r}} \\
		& =
		\frac{\gcd(a_{1}, \dotsc, a_{r})}{a_{1} \dotsc a_{r}} 
		\left[
			\sum_{p_{1} + \dotsc + p_{r} = n} \left(\frac{\alpha_{1}}{a_{1}}\right)^{p_{1}} \dotsc \left(\frac{\alpha_{1}}{a_{r}}\right)^{p_{r}} \right]
		\frac{m^{n + r -1}}{(n+ r - 1)!}
		+ o(m^{n+r-1}).
	\end{align*}
\end{proof}

\begin{remark} In \cite[Proposition 3.3.8]{Cad16}, the author invoked Toën's orbifold Riemann-Roch theorem to give an asymptotic estimate of a particular case of Lemma~\ref{lem:asymptfraction}, which might seem a bit disproportionate. Let us use these notes to give here a more down to earth argument. We fix \(k, n \in \mathbb{N}\). Identifying \(\alpha_{1} = \dotsc = \alpha_{r} = 1\) in the expression of Lemma~\ref{lem:asymptfraction}, and taking \(r = k\), we get:
	\begin{align*}
		\sum_{l_{1} + 2 l_{2} +  \dotsc + k l_{k} = m} \frac{(l_{1} + \dotsc + l_{k})^{n}}{n!}
		& = \left( 
		\frac{1}{k!} 
		\sum_{p_{1} + \dotsc + p_{k} = n} 
		\frac{1}{1^{p_{1}}} \dotsc \frac{1}{k^{p_{k}}}
		\right) \frac{m^{n+k-1}}{(n+k-1)!} + o(m^{n+k-1}) \\
		& = \frac{1}{k!} \left[ \sum_{1 \leq i_{1} \leq \dotsc \leq i_{n} \leq k} \frac{1}{i_{1} \dotsc i_{n}}\right] \frac{m^{n+k-1}}{(n+k-1)!} + o(m^{n+k-1}).
	\end{align*}
	The formula holds without restriction of divisibility on \(m\) since \(\gcd(1, 2, \dotsc, k) = 1\). This gives back the estimate of \cite{Cad16}.
\end{remark}

\subsection{Asymptotics of Euler characteristics}

In this section, we use Lemma~\ref{lem:asymptfraction} to determine the asymptotic behaviour of the Euler characteristics of symmetric powers of some weighted projective sums. 
\medskip

\begin{proposition}
	Let \(X\) be a complex projective manifold of dimension \(n\). Let \(L_{1}, \dotsc, L_{r}\) be line bundles on \(X\), and fix integers \(a_{1}, \dotsc, a_{r} \in \mathbb{N}\). Then, one has the asymptotic expansion, as \(m\) goes to \(+\infty\) while being divisible by \(\gcd(a_{1}, \dotsc, a_{r})\):
\[
\chi(X, S^{m}(L_{1}^{(a_{1})} \oplus \dotsc \oplus L_{r}^{(a_{r})}))
	 =
	\frac{\gcd(a_{1}, \dotsc, a_{r})}{a_{1} \dotsc a_{r}}
	\int_{X}
	\left(
		\prod_{j = 1}^{r} \sum_{p=0}^{n} \frac{c_{1}(L_{j})^{p}}{a_{j}^{p}}
	\right)_{n}
	\frac{m^{n+r-1}}{(n+r-1)!} + o(m^{n+r-1}).
\]
	where \((\cdot)_{n}\) means we take the part of pure degree \(n\) of the class between the brackets.
\end{proposition} \label{prop:symmetricsum}
 Let \(\alpha_{i} = c_{1}(L_{i})\) for all \(1 \leq i \leq r\). One has then, using the Hirzebruch-Riemann-Roch theorem:
\begin{align*}
	\chi(X, S^{m}(L_{1}^{(a_{1})} \oplus \dotsc \oplus L_{r}^{(a_{r})}))
	& = \chi(X, \bigoplus_{a_{1} l_{1} + \dotsc + a_{r} l_{r} = m} 
	L_{1}^{\otimes l_{1}} \otimes \dotsc \otimes L_{r}^{\otimes l_{r}}) \\
	& = \sum_{a_{1} l_{1} + \dotsc + a_{r} l_{r} = m} 
	\chi(X, L_{1}^{\otimes l_{1}} \otimes \dotsc \otimes L_{r}^{\otimes l_{r}}) \\
	& = \sum_{a_{1} l_{1} + \dotsc + a_{r} l_{r} = m} 
	\left[
	    \int_{X} \frac{(\alpha_{1} l_{1} + \dotsc + \alpha_{r} l_{r})^{n}}{n!} 
	    +
	    \sum_{j=0}^{n} \int_{X} \beta_{j} \cdot (\alpha_{1} l_{1} + \dotsc + \alpha_{r} l_{r})^{n-j}
	\right].
\end{align*}
where for all \(j = 1, \dotsc, n\), the symbol \(\beta_{j} \in H^{2j}(X)\) denotes a cohomology class depending only on \(X\), but not on \(m\). One can now apply Lemma~\ref{lem:asymptfraction} to obtain the result.

\bibliographystyle{amsalpha}
\bibliography{biblio}

\end{document}